\documentclass[a4paper,12pt, english]{article}
\usepackage[utf8x]{inputenc}
\usepackage{amsmath, amsthm, amssymb, graphicx, makeidx, mathrsfs, enumerate}
\usepackage{setspace}
\usepackage{hyperref}
\usepackage{multirow}
\usepackage[T1]{fontenc}
\usepackage{babel}
\usepackage{authblk}

\newcommand{\Z}{\mathbb Z}	% For Integers
\newcommand{\Q}{\mathbb Q}	% For Rational numbers
\usepackage{vmargin}
\setmarginsrb     { 1.0in}  % left margin
                        { 0.8in}  % top margin
                        { 0.9in}  % right margin
                        { 1.0in}  % bottom margin
                        {  20pt}  % head height
                        {0.25in}  % head sep
                        {   9pt}  % foot height
                        { 0.3in}  % foot sep
\newtheorem{theorem}{Theorem}[section]

\newtheorem{lemma}[theorem]{Lemma}

\newtheorem{remark}[theorem]{Remark}

\newtheorem{example}[theorem]{Example}

\newcommand{\PKm}{p_{i}^{k_{i,m}}}

\newcommand{\PKA}{p^{k}}

\newcommand{\NPMA}{n'}

\newcommand{\TNPIA}{\theta^{n'}}
\newcommand{\ITA}{\eta_{k}}
\newcommand{\ITPA}{\eta_{k}^{p^k}}

\newcommand{\TNPMJA}{\theta^{m}}
\newcommand{\CNPMJA}{C_{m}}
\newcommand{\ANPMJA}{\delta_{m}}

\usepackage{etoolbox}
\patchcmd{\thebibliography}
  {\settowidth}
  {\setlength{\parsep}{0pt}\setlength{\itemsep}{0pt plus 0.1pt}\settowidth}
  {}{}

\usepackage{blindtext}
\usepackage{sectsty}
\sectionfont{\centering}

\makeindex
\begin{document}

\pagenumbering{arabic}
\author[]{Anuj Jakhar}

\author[]{{\small {ANUJ JAKHAR\footnote{The Institute of Mathematical Sciences, HBNI, CIT Campus, Taramani, Chennai - 600113, Tamil Nadu, India. Email : anujjakhar@iisermohali.ac.in}  ~~ SUDESH K. KHANDUJA\footnote{Corresponding author.} \footnote{ Indian Institute of Science Education and Research Mohali, Sector 81, Knowledge City, SAS Nagar, Punjab - 140306, India  \textsc{\&} Department of Mathematics, Panjab University, Chandigarh - 160014, India. Email : skhanduja@iisermohali.ac.in} ~~AND~ NEERAJ SANGWAN\footnote{Indian Institute of Technology (IIT), Bombay,  Mumbai-400076, India, neerajsan@iisermohali.ac.in} }}}

\date{}
\renewcommand\Authands{}

\title{\textsc{On integral basis of pure number fields}}

\maketitle
\begin{center}
{\large{\bf {\textsc{Abstract}}}}
\end{center}  Let $K=\mathbb{Q}(\sqrt[n]{a})$ be an extension of degree $n$ of the field $\Q$ of rational numbers, where the integer $a$ is such that for each prime $p$ dividing $n$ either $p\nmid a$ or the highest power of $p$ dividing $a$ is coprime to $p$; this condition is clearly satisfied when $a, n$ are coprime or $a$ is squarefree. The present paper gives explicit construction of  an integral basis of $K$ along with applications. This construction of an integral basis of $K$ extends a result proved in [J. Number Theory, {173} (2017), 129-146] regarding periodicity of integral bases of  $\mathbb{Q}(\sqrt[n]{a})$ when $a$ is  squarefree. %This leads to an innovative method of constructing algebraic integers in $\Q(\sqrt[n]{a})$.%Although an integral basis (which yields a formula for the discriminant) of $\mathbb{Q}(\sqrt[n]{a})$ with $a, n$ coprime is derived in a series of five papers culminating in [Proc. Japan Acad. 58A (1982) 219-222], we give counter examples to show that this formula is incorrect.  

\bigskip


\bigskip

\newpage
\section{\textsc{Introduction and Statement of Results}}
 Discriminant  is  a valuable tool to find a $\Z$-basis for the ring $A_K$ of
algebraic integers in an algebraic number field $K$. To describe a $\Z$-basis of $ A_K$ (called an integral basis of $K$) is in general a difficult task.
   The problem of  computation of integral basis specially for pure 
algebraic number fields has attracted the attention of several mathematicians. In 1900, Dedekind \cite{Ded} described an integral basis of  pure cubic fields.   Westlund  \cite{Wes} in 1910 gave an integral basis for all fields  of the type $\Q(\sqrt[p]{a})$ having prime degree $p$ over  $\Q$.  
 In 1984, Funakura \cite{Fun} 
determined an integral basis of all pure quartic fields. In  2015, Hameed and Nakahara  \cite{HN} provided an integral basis of those pure octic fields $\Q(\sqrt[8]{a})$, where $a$ is a squarefree integer. In 2017,  Ga\'{a}l and  Remete \cite{Ga-Re} gave a construction of integral basis of $K= \Q(\sqrt[n]{a})$ where the integer $a$ is squarefree and $3\leq n\leq 9$; they further proved that if $a, a'$ are squarefree integers which are congruent modulo $n_0^n$, $n_0$ being the largest square dividing $n^n$ and if $\theta, \theta'$ are roots of the irreducible polynomials $x^n - a$, $x^n - a'$ respectively, then an element $\frac{1}{q}(c_0 + c_1\theta + \cdots + c_{n-1}\theta^{n-1})$ with $c_i, q\in \Z$, is an algebraic integer if and only if so is $\frac{1}{q}(c_0 + c_1\theta' + \cdots + c_{n-1}\theta'^{n-1})$. This result they expressed by saying that the integral bases of the fields  $\Q(\sqrt[n]{a})$  are periodic in $a$ with period $n_0^n$ when $a$ is squarefree. % In  this paper, our aim is to explicitly construct an integral basis of $n$-th degree fields of the type $\Q(\sqrt[n]{a})$  where for each prime $p$ dividing $n$, either $p$ does not divide $a$ or the highest power of $p$ dividing $a$ (to be denoted by $v_p(a)$) is coprime with $p$. The above conditions clearly takes care of all number fields of the type $\Q(\sqrt[n]{a})$, where either $a,n$ are coprime or $a$ is squarefree. Though a description for an  integral basis of fields of the type $\Q(\sqrt[n]{a})$ with $a,n$ coprime is derived in a series of papers  (\cite{KOK1},\cite{KOK2},\cite{KOK3},\cite{KOK4},\cite{KOK}), we give   counter examples to show that this is incorrect $($see Examples 3.1, 3.2$)$.  Moreover our construction of integral basis is practicable. At the end, we illustrate our method of computation of integral basis through examples. 

In the present paper, our aim is to give an explicit construction of an integral basis of pure fields of the type $K=\Q(\sqrt[n]{a})$ where for each prime $p$ dividing $n$, either $p$ does not divide $a$ or the highest power of $p$ dividing $a$ is coprime to $p$; this condition is satisfied when either $(a,n)=1$ or $a$ is squarefree. Moreover, we shall quickly deduce from our construction that the integral bases of the fields $\Q(\sqrt[n]{a})$, $a$ being squarefree, are periodic in $a$ with period $np_1\cdots p_k$, where $p_1, \cdots, p_k$ are the distinct primes dividing $n$. We shall use the following theorem proved in $(\cite{Gass}, \cite{JKS26})$.
%\section{\textsc{Integral basis of $\Q(\sqrt[n]{a}).$}}

\noindent{\bf Theorem 1.A.} {\it  Let $K=\mathbb Q(\theta)$ be an algebraic number field  with discriminant $d_K$,  where $\theta$ is a root of an irreducible polynomial  $f(x) = x^{n} - a$ belonging to $\mathbb Z[x]$.  Let $\prod\limits_{i=1}^k p_i ^{s_i}, \prod\limits_{j=1}^l q_{j}^{t_j}$ be the prime factorizations of $n,|a|$ respectively. Let $m_j, n_i$ and $r_i$ stand respectively for the integers $\gcd(n, t_j), \frac{n}{p_i^{s_i}}$ and $ v_{p_i} (a^{p_i -1} -1)-1$. Assume that $a$ is $n$-th power free and for each $i$,  either $v_{p_i}(a)=0 $ or $v_{p_i}(a)$ is coprime to $p_i$.  Then $$d_K 
=(-1)^{\frac{(n-1)(n-2)}{2}} sgn(a^{n-1})(\prod\limits _{i=1}^k p_i ^{ v_i}) 
\prod\limits_{j=1}^l q_j ^{n-m_j}, 
$$ where   $v_i $ equals  $ns_i -2n_i \sum\limits_{j=1}^{\min\{r_i,s_i\} }  p_i^{s_i 
	-j}$ or $ns_i $ according as  $r_i>0$ or not.} 
%\end{theorem}

The basic idea of our construction of integral basis $\Q(\sqrt[n]{a})$ has been derived from two  lemmas which are stated below and will be proved in the next section. These lemmas lead to a method to construct certain algebraic integers in $K$ which play a significant role in the formation of integral basis of $K$. The proofs of all our results in the paper are based on elementary algebraic number theory and are self contained. 
 In what follows, $\theta$ is a root of 
an irreducible polynomial $x^n - a \in \Z[x]$, where $|a|$ is $n$-th power free which is expressed as 
$\prod\limits_{j=1}^{n-1}a_j ^j$ where $a_j$'s are squarefree numbers which are relatively prime in 
pairs. For a real number $\lambda$, $\lfloor \lambda \rfloor$ will stand for the greatest integer not exceeding $\lambda$.  For $0 \leq m \leq n-1$, $C_m$ will stand for the number $\prod\limits_{j=1}^{n-1}a_j 
^{\lfloor\frac{jm}{n}\rfloor}$. Keeping in mind that $\theta^n = a$, it can be easily seen that the element 
$\frac{\theta^m}{C_m}$ satisfies the polynomial $x^n - sgn(a)^m\prod\limits_{j=1}^{n-1}a_{j}^{jm-
	n\lfloor\frac{jm}{n}\rfloor}$ and hence is an algebraic integer.

With above notation, we shall prove the following two lemmas in the next section. 
\begin{lemma} \label{3.3}
	Let $K = \Q(\theta)$ with $\theta$  a root of an irreducible polynomial $x^n -a \in \Z[x]$. Let $s\geq 1$ be 
	the highest power of a prime number $p$ dividing $n$. Assume that $p$ does not divide $a$ and the integer $r = v_p(a^{p-1}-1)-1$ is positive. Corresponding to a number $k$, $1 \leq k \leq \min\{r, s\}$, let  $b'$ be an integer satisfying $ab' 
	\equiv 1~(mod~p^{k+1})$ and $a'$ be given by $(b')^{p^{s-k-1}}$ or $b'$ according as $k < s$ or not. Let $\NPMA$ denote the integer $\frac{n}{p^k}$.  If $\ITA$ 
	is the element  $ \sum\limits_{j=0}^{\PKA-1}\left(a'\TNPIA\right)^{j}$ of $K$, then $\dfrac{\ITA}{\PKA}$ is an algebraic integer.
\end{lemma}

\begin{lemma}\label{3.4}
	Let $x^n -a,  p, s$ and $r$ be as in Lemma \ref{3.3}. For any given integer $m, 0 \leq m \leq n-1$, let $k_m$ be the largest non-negative integer not exceeding $\min\{r,s\}$ such that $m = n - \frac{n}{p^{k_m}} + j_m$ with $j_m \geq 0$. Let  $b'_m$ be an integer satisfying $ab'_m 
	\equiv 1~(mod~p^{k_m+1})$ and $a'_m$ be given by $(b'_m)^{p^{s-k_m-1}}$ or $b'_m$ according as $k_m < s$ or not.  Let $n'_m$ stand for the integer $\frac{n}{p^{k_m}}$ and $w_{m}$ be an integer such that  $w_{m}\CNPMJA (a'_m)^{p^{k_m} - 1} \equiv  1 (mod ~p^{k_m}).$ Let $\ANPMJA$ stand for the element  $ w_{m}\CNPMJA\theta^{j_{m}} \sum\limits_{j=0}^{p^{k_m}-2}\left(a'_m\theta^{n'_m}\right)^{j} $ or $0$ of $\Z[\theta]$ according as $k_m > 0$ or $k_m = 0$, then $\dfrac{\TNPMJA + \ANPMJA}{p^{k_m}~ \CNPMJA}$ is an algebraic integer.
\end{lemma}

For reader's convenience we first describe an explicit integral basis of $K$  (with $K$ as in Theorem 1.A)  when its degree is power of a prime.
\begin{theorem}\label{3.3'}
	Let $K = \Q(\theta)$ with $\theta$  a root of an irreducible polynomial $x^{p^s} - a $ belonging to $ \Z[x]$ of degree power of a prime $p$, where $|a|$ is $p^s$-th power free which is expressed as
	$\prod\limits_{j=1}^{p^s-1}a_j^j$, $a_j$'s being squarefree numbers relatively prime in pairs. Assume that $p$ does not divide the $\gcd$ of $a$ and $v_p(a)$. Let $r, d$ denote respectively the integers
	$v_p(a^{p-1}-1)-1$ and $\min\{r, s\}$.  For $0\leq m <p^s$, let $k_m$ stand for the largest non-negative integer not exceeding $d$ such that $m \geq p^s - p^{s-k_m}$ and  $j_m$ be given by $m = p^s - p^{s-k_m} + j_m$. Let $C_m$ stand for $\prod\limits_{j=1}^{p^s-1}a_j^{\lfloor\frac{jm}{p^s}\rfloor}$. 
	When  $m\geq \phi(p^s)$\footnote{Note that  $m\geq \phi(p^s)$  if and only if $k_m \geq 1$. }, fix any integer $a'_m$ congruent to $a^{(p-2)(p^{s-k_m-1})}$ or $a^{p-2}$ modulo $p^{k_m+1}$ according as $k_m < s$ or $k_m = s$. 
	Let $w_m$ belonging to $\Z$ be such that  $w_m C_{m} (a'_{m})^{p^{k_m}-1}\equiv 1~(mod~p^{k_m})$ and $\delta_m$ denote the element $w_mC_m\theta^{j_m}\sum\limits_{i=0}^{p^{k_m}-2}(a'_m\theta^{p^{s-k_m}})^i$ when $m\geq \phi(p^s)$. Then the following hold:\\
	$(i)$ If $r = 0$ or $-1$, then $\{1, \frac{\theta}{C_1}, \frac{\theta^2}{C_2}, \cdots, \frac{\theta^{p^s-1}}{C_{p^s-1}}\}$ is an integral basis of $K$.\\
	$(ii)$ If $r \geq 1$, then $\{1, \frac{\theta}{C_1}, \frac{\theta^2}{C_2}, \cdots, \frac{\theta^{\varphi{(p^s)}-1}}{C_{\varphi(p^s)-1}}\} \cup \{\frac{\theta^m + \delta_m}{p^{k_m}C_m} ~|~ \varphi{(p^s)} \leq m \leq p^s-1\}$ is an integral basis of $K$.
\end{theorem}

%The following corollary which gives an integral basis of all pure prime degree number fields is an immediate consequence of the above theorem.
%\begin{corollary}Let $K = \Q(\theta)$ with $\theta$  a root of an irreducible polynomial $x^{p} - a $ belonging to $ \Z[x]$ of prime degree $p$, where 
% $|a|$ is $p^{th}$ power-free which is expressed as $\prod\limits_{j=1}^{p-1}a_{j}^{j}$,  $a_j$'s being squarefree numbers relatively prime in pairs. Let $C_m$ denote the integer $\prod_{j=1}^{p-1}a_j^{\lfloor \frac{jm}{p}\rfloor }$ and $\delta$ denote the element $ C_{p-1} ^{p-1} \sum_{i=0}^{p-2} (a^{p-2}\theta)^i$ of $\Z[\theta]$. Then the following hold: \\
%	(i) If  $p^{2}\nmid (a^{p-1}- 1)$, then the set $\{1,  \frac{\theta}{C_1}, \frac{\theta^2}{C_2}, \cdots, \frac{\theta^{p-1}}{C_{p-1}} \} $ is an integral basis of $K$. \\
%	(ii) If $p^2 \mid (a^{p-1}- 1)$, then  $\{1,  \frac{\theta}{C_1}, \frac{\theta^2}{C_2}, \cdots, \frac{\theta^{p-2}}{C_{p-2}},\frac{\theta^{p-1}+ \delta}{pC_{p-1}} \}$  is an integral basis of $K$.
%\end{corollary}
Theorem \ref{3.3'} will be deduced as a corollary to our main theorem which will be stated after giving some examples illustrating Theorem \ref{3.3'}.

\begin{example} Let $K = \Q(\theta)$, where  $\theta$ is a root of $ x^8 - 17\cdot 7^2$ and $A_K$ be the ring of algebraic 
	integers of $K$.
	% With notation as in  \cite[pp. 221-222]{KOK}, the polynomial $g_m(x)$ can be chosen to be $(x-17\cdot 7^2)^m$ for $0\leq m\leq 7$. So by the main result of $\cite{KOK}$, $\{1,\theta-17\cdot 7^2,(\theta-17\cdot 7^2)^2, (\theta-17\cdot 7^2)^3,\frac{1}{14}(\theta-17\cdot 7^2)^4,\frac{1}{14}(\theta-17\cdot 7^2)^5,\frac{1}{14}(\theta-17\cdot 7^2)^6,$ $	\frac{1}{14}(\theta-17\cdot 7^2)^7 \}$ is an integral basis of $K$ and accordingly $[A_K : \Z[\theta]]$ = $14^4$. 	We obtain the desired contradiction by showing that $2^7 7^4$ divides $[A_K:\Z[\theta]]$.
	We retain the notations of Theorem \ref{3.3'}. 
	Here  $ a = 17\cdot 7^2, p = 2, s = 3, r \geq 5$ as $a\equiv 1~(mod~ 2^6)$. One can verify that $C_m=1 \text{ or } 7 $ according as $0\leq m\leq 3$ or $4\leq m\leq 7$. It can be easily checked that $k_{0} = k_{1} = k_{2} = k_{3} = 0,$ 
	$k_{4} = k_{5} = 1, k_{6} = 2,k_{7} = 3, j_4=0, j_5=1, j_6=0$ and $ j_7=0$. Keeping in mind $a\equiv 1 ~(mod ~2^6)$ we see that  $a'_{m} = 1$ works for $0\leq m \leq 7$. 
	Accordingly one can choose $w_{4} = w_{5} = 1$, $w_{6} = 3$ and $w_{7} = 7$. Therefore with $\delta_m$ as defined in Theorem \ref{3.3'}, we have  $\delta_4=7, \delta_5= 7\theta, \delta_6 = 21(\theta^4 +\theta^2+ 1)$ and $\delta_7= 49(\sum_{j=0}^{6} \theta^j)$. Substituting for $C_m, k_m$ and $\delta_m$ in the integral basis described in   Theorem $\ref{3.3'}$, we see that  the set  
	%	$\{\frac{\theta^m + \delta_{m}}{p^{k_{m}} C_m} \vert 0 \leq m \leq 7\} $  which  equals 
	$\{1,\theta,
	\theta^2,\theta^3,\frac{\theta^4+7}{14},\frac{\theta^5+7\theta}{14},\frac{\theta^6+21(\theta^4 +\theta^2 +1)}
	{28}, \frac{\theta^7 + 49(\sum\limits_{j=0}^{6} \theta^j)}{56} \}$  is an integral basis of $K$. 
\end{example}

\begin{example} Let $K = \Q(\theta)$ with  $\theta$ a root of $x^9 - 26$. %According to $\cite{KOK}$, $\{1,\theta-26,(\theta-26)^2, (\theta-26)^3,(\theta-26)^4,(\theta-26)^5,\frac{1}{3}(\theta-26)^6,\frac{1}{3}(\theta-26)^7,\frac{1}{3}(\theta-26)^8 \}$ is an integral basis of $K$ which implies that $[A_K : \Z[\theta]]$ = $3^3$. We show that the latter assertion is false by proving that $3^4$ divides $ [A_K:\Z[\theta]]$.
	With notations as in Theorem \ref{3.3'}, here $ p = 3, s = 2, a = 26$ and $r\geq 2$. Since $a$ is a squarefree integer, $|a|=a_1$ and hence $C_m = a_1^{\lfloor \frac{m}{9}\rfloor} =1$ for each $m\leq 8$. It can be easily checked that $k_{m} = 0$ for $0\leq m \leq 5$, $k_{6} = k_{7} = 1,~k_{8} = 2,~ j_6=0, ~j_7=1$ and $j_8=0$. Here $a'_{m} = -1$  works for $m = 6, 7, 8$ and accordingly we may choose $w_{m} = 1$ for these values of $m$. Thus we have $\delta_6 = -\theta^3 +1, \delta_7 = -\theta^4 +\theta$ and $\delta_8 = \sum_{j=0}^7 (-\theta)^j$. Therefore by Theorem \ref{3.3'},  
	$\{1,\theta,\theta^2,\theta^3,\theta^4,\theta^5,\frac{\theta^6 - \theta^3 +1}{3}, \frac{\theta^7 - \theta^4 +\theta}{3},\frac{\theta^8 + \sum\limits_{j=0}^{7} (-\theta)^j }{9} \}$  is an integral basis of $K$.
\end{example}

Before explicitly stating an integral basis of $\Q(\sqrt[n]{a})$, we introduce some notation. \\
Let $K=\mathbb Q(\theta)$ be an algebraic number field with $\theta$  a root of an irreducible polynomial  $ x^{n} - a$ belonging to $\mathbb Z[x]$ where $|a|$ is a $n$-th power free integer which is expressed as $\prod_{j=1}^{n-1} a_j^j$, $a_j$ being squarefree numbers relatively prime in pairs.  Let $n=\prod\limits_{i=1}^k p_i ^{s_i}$, $|a|= \prod_{j=1}^l q_j^{t_j}$ be the prime factorizations of $n$, $|a|$ and let  $r_i$ be as in Theorem $1.A$.   For $0 \leq m \leq n-1$,
$C_m$ will stand for the number $\prod\limits_{j=1}^{n-1}a_j^{\lfloor\frac{jm}{n}\rfloor}$. We shall denote by 
$S$ the set $\{i  ~|~  1\leq i \leq k, r_i \geq 1\}$. For $i \in S$,  denote $\min\{r_i, s_i\}$ by $d_i$. Note that for such an index $i$, $p_i$ does not divide $a$, for otherwise $r_i = v_{p_i}(a^{p_i-1}-1)-1$ = $-1 < 0$. Since the interval $[0, n-1]$ can be 
partitioned into the union $\cup_{k_i = 0}^{d_i - 1}[n - \frac{n}{p_i^{k_i}}, n - \frac{n}{p_{i}^{k_i+1}}) \cup [n - 
\frac{n}{p_i^{d_i}}, n-1]$ of pairwise disjoint intervals, it follows that given any non-negative integer $m$ not 
exceeding $n-1$, it belongs to one of these intervals and hence there exists a largest integer $k_{i,m}, 0 \leq k_{i,m} \leq d_i$ and a non-negative integer 
$j_{i,m}$ such that $m = n - \frac{n}{p_i^{k_{i,m}}} + j_{i,m}$. For $0\leq m\leq n-1$, let $S_m$ denote the subset of $S$ defined by $S_m = \{i \in S ~|~ k_{i,m} \geq 1\}$. Fix a pair $(i,m)$ with $i\in S_m$.  Choose an integer $b'_{i,m}$ such that $ab'_{i,m} 
\equiv 1~\mbox{mod}(p_i^{k_{i,m}+1})$ and set $a'_{i,m} = (b'_{i,m})^{p_i^{s_i-1-k_{i,m}}}$ or $b'_{i,m}$ according as 
$1\leq k_{i,m} < s_i$ or $k_{i,m} = s_i$. Corresponding to the pair $(i,m)$, denote the integer $ 
\frac{n}{p_i^{k_{i,m}}}$ by $n_{i,m}$. Let  $w_{i, m}$  be an integer such that  $w_{i, m}C_{m} 
(a'_{i,m})^{p_i^{k_{i,m}}-1} \equiv 1 (mod~ p_i^{k_{i,m}}) $; such $w_{i,m}$ exists because $p_i \nmid 
C_ma'_{i,m}$.  
 Define $\delta_{i,m}$ belonging to $\Z[\theta]$ by $\delta_{i,m} = w_{i,m}C_m\theta^{j_{i,m}}\sum\limits_{r=0}^{p_i^{k_{i,m}}-2}(a'_{i,m}\theta^{n_{i,m}})^{r}$.
 Let $z_{i,m}$ 
stand for the number $\prod\limits_{j\in S_m\backslash\{i\}}p_j^{k_{j,m}}$ for $i \in S_m$. Since $\gcd\{z_{i,m} ~|~ i \in S_m\}$ = 1, there exist integers $u_{i,m}$ such that 
$\sum\limits_{i\in S_m}u_{i,m} z_{i,m} = 1$. Let $\beta_m$ equal
$ \sum\limits_{i \in S_m}u_{i,m} z_{i,m}\delta_{i,m}$ or $0$ according as $S_m \neq \emptyset$ or $S_m = \emptyset$.  Observe that the highest power of $\theta$ occuring in $\delta_{i,m}$ is $j_{i,m}+n-2n_{i,m} = m-n_{i,m}<m$ and hence the same holds for $\beta_m$. 
Note that when $S_m$ is a singleton set consisting of $\{i\}$, then $z_{i,m}$ being an empty product is $1$ and $\beta_m= \delta_{i,m}$.

%or $0$ according as $m \geq n - \frac{n}{p_i}$ for some $i \in S$ or not. 
\indent With the above notations, we shall  prove the following theorem. \vspace*{-3mm}
\begin{theorem}\label{3.1} Let $K = \Q(\theta)$ with $\theta$ having minimal polynomial $x^n - a$ over $\Q$ where 
$a$ is an $n$-th power free integer and for every prime $p_i$ dividing $n$ either $p_i \nmid a$ or $v_{p_i}(a)$ is 
coprime to $p_i$. Let $n = \prod\limits_{i=1}^{k}p_i^{s_i}, S, C_m, k_{i,m}$ and $\beta_m$ be as in the 
above paragraph.  Then the following hold:\\
  $(i)$ If $S=\emptyset$, then $ \left\lbrace  \dfrac{\theta^m}{C_m } ~\bigg\vert~ 0\leq m\leq n-1  \right\rbrace $ is an integral basis of $K$.   \\
  $(ii)$ If $S\neq \emptyset$, then $ \left\lbrace 1,  \dfrac{\theta^m +\beta_m}{C_m \prod\limits_{i\in S} p_i ^{k_{i,m}}} ~\bigg\vert~ 1\leq m\leq n-1  \right\rbrace $ is an integral basis of $K$. 
\end{theorem}      

\begin{remark}\label{3.5'} In the paragraph preceding Theorem \ref{3.1}, note that $b'_{i,m}$ can be chosen to be any integer which is congruent to $a^{p_i-2}$ modulo  $p_i^{k_{i,m}+1}$ because $a^{p_i -1} \equiv 1~(mod~p_i^{r_i+1})$ and $k_{i,m} \leq r_i$.% It may also be pointed out that by virtue of Euler's Theorem, $w_{i,m}$ can be taken to be  any integer which is congruent to $\left(C_{m} (a'_{i,m})^{p_i^{k_{i,m}}-1}\right)^{\varphi(p_i^{k_{i,m}})-1}$ modulo $p_i^{k_{i,m}}$. 
	%$(i)$   $b'_{i,m}$  can be chosen such that $b'_{i,m} \equiv a^{p_i^{s_i -2}} (mod~ p_i^{k_{i,m}+1})$; \\
	%$(ii)$  $w_{i,m}$ can be chosen such that  $w_{i,m} \equiv \left(C_{m} (a'_{i,m})^{p_i^{k_{i,m}}-1}\right)^{p_i^{k_{i,m}+1} - p_i^{k_{i,m}}-1}~(mod~p_i^{k_{i,m}+1})$.
\end{remark}

The following theorem which improves the results of Theorems 2, 3 of \cite{Ga-Re} will be quickly deduced from the above theorem.
\begin{theorem}\label{3.new}
Let $n\geq 2$ be an integer with prime divisors $p_1, \cdots, p_k$. Let $K = \Q(\theta), K' = \Q(\theta')$ be algebraic number fields of degree $n$, where $\theta, \theta'$ satisfy respectively the polynomials $x^n - a$, $x^n - a'$  with integers $a,a'$ squarefree which are congruent modulo $np_1\cdots p_k$, then an element $\frac{1}{q}(c_0 + c_1\theta + \cdots + c_{n-1}\theta^{n-1})$ with $c_i, q \in \Z$ is in $A_K$ if and only if $\frac{1}{q}(c_0 + c_1\theta' + \cdots + c_{n-1}\theta'^{n-1})$ is in $A_{K'}.$
\end{theorem}

\noindent We now provide some examples to illustrate Theorem \ref{3.1}.

\begin{example} Let $\theta$ be a root of $x^{10} - 150$ and $K = \Q(\theta)$. With notations  as in Theorem $\ref{3.1}$, it can be easily seen that here $S=\emptyset$,  $C_m=1$ for $0\leq m\leq 4$ and $C_m=5$ for $5\leq m \leq 9$.   Hence $\{1,\theta,\theta^2,{\theta^3},\theta^4,\frac{\theta^5}{5}, \frac{\theta^6}{5}, \frac{\theta^7}{5}, \frac{\theta^8}{5}, \frac{\theta^9}{5}\}$ is an integral basis of $K$ by Theorem $\ref{3.1}(i)$. \end{example}

\begin{example} Let $K = \Q(\theta)$, where  $\theta$ is a root of $x^6 - 2\cdot 5^2\cdot 13^5$. We retain the notations introduced  in the paragraph preceding Theorem $\ref{3.1}$ and take $p_1 = 3, p_2 = 2, a=2\cdot 5^2\cdot 13^5$. One can verify that $C_1 = 1, C_2 = 13, C_3 = 5\cdot 13^2, C_4 = 5\cdot 13^3, C_5 = 5\cdot 13^4$. As can be easily seen $v_3(a^2 - 1) = r_1 + 1 \geq 2$ and $v_2(a-1) = r_2+1 = 0$. Therefore $S = \{1\}$. One can quickly verify that $k_{1,m} = 0$ for $0 \leq m \leq 3$ and $k_{1,m} = 1$ for $m = 4,5$; also $j_{1,4}=0, j_{1,5} =1$. Here $b'_{1, m} = -1$ works for $m = 4, 5$ and in view of this choice we may take $a'_{1,m} = -1$ for $m = 4,5$. We may take $w_{1,m}=2$ for $m=4,5$. Since $S=\{ 1\}$ is a singleton, we have $\beta_m= \delta_{1,m}$ for $m=4,5$.  Substituting these values, we see that $\beta_4= 10\cdot 13^3 (-\theta^2 +1) $ and $\beta_5=10\cdot 13^4  (-\theta^3 
	+\theta)$. Therefore by Theorem $\ref{3.1}(ii)$ that $\{1,\theta,\frac{\theta^2}{13},\frac{\theta^3}{5\cdot 
 13^2},\frac{\theta^4 + 10\cdot 13^3 (-\theta^2 +1) }{15 \cdot 13^3},\frac{\theta^5+ 10\cdot 13^4  (-\theta^3 
 +\theta)}{15 \cdot 13^4} \}$ is an integral basis of $K$. \end{example}
 
\begin{example} Let $\theta$ be a root of $x^{6} - 37$ and $K = \Q(\theta)$. We retain the notations  introduced in the paragraph preceding Theorem $\ref{3.1}$ we take $p_1 = 3, p_2 = 2$ and $a=37$. It is clear that  $v_3(a^2 - 1) = r_1 + 1 = 2$ and $v_2(a-1) = r_2+1 = 2$. Therefore $S = \{1, 2\}$. One can easily verify 
 that $k_{1,m} = 0$ for $0 \leq m \leq 3$, $k_{1,m} = 1$ for $m = 4,5$, $j_{1,4}=0 $ and $j_{1,5}=1$. Here $b'_{1, m} = 1$ works for $m = 4, 
 5$ and so we may choose $a'_{1,m} = 1$ for $m = 4,5$. Since $a$ is squarefree, $C_m = 1$ for all $m$; we 
 may take $w_{1,m} = 1$ for $m = 4,5$. So $\delta_{1,4} = \theta^2 + 1$ and $\delta_{1,5} = \theta^3+\theta$. 
 Further $k_{2,m} = 0$ for $0 \leq m \leq 2$, $k_{2,m} = 1$ for $m = 3,4,5$, $j_{2,3}=0,~j_{2,4}=1 $ and $j_{2,5}=2$. Here $b'_{2, m} 
 = 1$ works for $m = 3, 4, 5$ and in view of this choice $a'_{2, m} = 1$ for $m = 3, 4,5$. As $C_m = 1$ for all 
 $m$; we may take $w_{2,m} = 1$ for $m = 3, 4, 5$. Thus here $\delta_{2,3} = 1, \delta_{2,4} = \theta$ and 
 $\delta_{2,5} = \theta^2$. Note that $S_m = \emptyset$ for $0\leq m\leq 2$, $S_3 = \{2\}$ and $S_m = \{1,2\}$ 
 for $m = 4,5$. Here $z_{2,3} = 1$, $z_{1,m} = 2$ for $m = 4,5$ and $z_{2,m} = 3$ for $m = 4,5$, therefore we can 
 take $u_{2,3} = 1$ and $u_{1,m} = -1, u_{2,m} = 1$ for $m = 4,5$. Note that $\beta_1=\beta_2 =0$, $\beta_3= \delta_{2,3}$ and $ \beta_m = u_{1,m}z_{1,m} \delta_{1,m}+ u_{2,m}z_{2,m} \delta_{2,m}$ for $m=4,5$. Substituting these values, it follows quickly 
 from Theorem $\ref{3.1}(ii)$ that $\{1,\theta, \theta^2,\frac{\theta^3+1}{2}, \frac{\theta^4 -2\theta^2 + 
 3\theta-2}{6},\frac{\theta^5-2\theta^3 + 3\theta^2-2\theta}{6} \}$ is an integral basis of $K$.  
\end{example}

\section{Proof of Lemmas \ref{3.3} and \ref{3.4}.}

\textit{Proof of Lemma \ref{3.3}.}
Multiplying $\ITA = 1 + (a' \TNPIA) + \cdots + (a'\TNPIA)^{\PKA-1}$ by  $(a' \TNPIA - 1)$ on both sides and using the fact that $\theta^{n'p^k} = a$, we see that
$$a'\TNPIA\ITA = \ITA+a(a')^{\PKA} - 1.$$
Now taking $\PKA$-th power on both sides, we obtain
$$a(a')^{\PKA} \ITPA = \sum\limits_{j=0}^{\PKA}\binom{\PKA}{j}\ITA^{p^{k}-j}\left(a(a')^{\PKA} - 1\right)^{j},$$
which  can be rewritten as 
\begin{equation}\label{eq:321a}
\left(a(a')^{\PKA} - 1\right)\ITPA = \sum\limits_{j=1}^{\PKA} \binom{\PKA}{j}\left(a(a')^{\PKA} - 1\right)^{j}\ITA^{p^{k}-j}.
\end{equation}
Note that  $a(a')^{\PKA} \neq 1$, for otherwise $a = \pm 1$ when $p$ is odd and $a = 1$ when $p$ = $2$ which is impossible in view of irreducibility of $x^n - a$ over $\Q$. On dividing $(\ref{eq:321a})$ by $a(a')^{\PKA} - 1$, we have
\begin{equation*}
\ITA^{\PKA} = \sum\limits_{j=1}^{\PKA}{\binom{\PKA}{j}\left(a(a')^{\PKA} - 1\right)^{j-1}}{\ITA}^{\PKA-j}.
\end{equation*}
Thus $\ITA$ satisfies the polynomial $x^{p^{k}} - \sum\limits_{j=1}^{\PKA}{\binom{\PKA}{j}\left(a(a')^{\PKA} - 1\right)^{j-1}}x^{\PKA-j}$; consequently $\dfrac{\ITA}{\PKA} $ satisfies the polynomial $x^{p^{k}} - \sum\limits_{j=1}^{\PKA}\dfrac{\binom{\PKA}{j}\left(a(a')^{\PKA} - 1\right)^{j-1}}{(\PKA)^{j}} x^{\PKA-j}$ = $h(x)$ (say).
In view of the following Lemma \ref{3.31}, $h(x) \in \Z[x]$ and so $\frac{\ITA}{p^{k}}$ is an algebraic integer as desired.
      
\begin{lemma}\label{3.31}
Let $x^n - a$, $p, s, r, k$ and $a'$ be as in the above lemma. Then $p^{jk}$ divides $\binom{\PKA}{j}\left(a(a')^{\PKA} - 1\right)^{j-1}$ for $1 \leq j \leq \PKA$.
\end{lemma}
\begin{proof}
We first prove that \begin{equation}\label{eq:3.31a}
a(a')^{p^k} \equiv 1~(mod~p^{k+1}).
\end{equation} 
Keeping in mind that $r+1  = v_p(a^p - a)$ and the fact that $k \leq r$, we have
\begin{equation}\label{eq:l2}
 a^{p} \equiv a~(mod~p^{k+1}).
\end{equation}
 The proof of $(\ref{eq:3.31a})$ is split into two cases.
%By the hypothesis we have
%\begin{equation}\label{eq:I2}
%a^{p_{i}^{s_i}} \equiv a~(mod~p_i^{r_i+1}).
%\end{equation}
%We divide the proof of the claim into two cases.
 Consider first the case when $k < s$. 
%Since $k \leq r$, in view of $(\ref{eq:l2})$ we have
%Also, using $(\ref{eq:I2})$, we have 
%\begin{equation}\label{eq:l5}
%a^{p^{s}} \equiv a~(mod~p^{k+1}).
%\end{equation}
It follows from $(\ref{eq:l2})$  that
$
a^{p^{s-1}} \equiv a~(mod~p^{k+1}).
$ Multiplying the last congruence by $(b')^{p^{s-1}}$ on both sides, we have $(ab')^{p^{s-1}} \equiv a(b')^{p^{s-1}}~(mod~p^{k+1})$. By choice $ab' \equiv~1~(mod~p^{k+1})$; consequently using the last two congruences, we see that $1 \equiv a((b')^{p^{s-k-1}})^{p^k}~(mod~p^{k+1})$. Recall that in the present case $a' = (b')^{p^{s-k-1}}$. So $a(a')^{p^k} \equiv 1~(mod~p^{k+1})$ proving $(\ref{eq:3.31a})$ in this case.
%Thus assertion $(i)$ follows  quickly in this case from the fact that $(ab')^{p^{s-1}} \equiv 1~(mod~p^{k+1})$, $a^{p^{s-1}} \equiv a~(mod~p^{k+1})$ and $a' = (b')^{s-k-1}.$
When $k = s$, then by choice $a' = b'$ which in view of the hypothesis  $ab' \equiv~1(mod~p^{k+1})$ and $(\ref{eq:l2})$  implies that  
 $1 \equiv (ab')^{p^k} \equiv a(a')^{p^k}(mod~p^{k+1})$ and hence  $(\ref{eq:3.31a})$ is proved.
%Therfore in view of the fact that $(ab')^{p^{s}} \equiv 1~(mod~p^{s+1})$, $a^{p^{s}-1} \equiv 1~(mod~p^{s+1})$ and $a' = b'$, assertion $(i)$ follows.\\
%Choose $b'_{i}$ such that 
%\begin{equation}\label{eq:l6'}
%ab'_{i} \equiv 1(mod~p_i^{r_i+1}),
%\end{equation}  which can be rewritten in $(mod~p_i^{k_i+1})$ as $ab'_{i} \equiv 1(mod~p_i^{k_i+1})$.  From the last congruence we obtain $$(ab'_{i})^{p_i^{s_i-1}} \equiv 1(mod~p_i^{k_i+1}).$$
%Using $(\ref{eq:l6})$, we see that
%$$a(b'_{i})^{p_i^{s_i-1}} \equiv 1(mod~p_i^{k_i+1}),$$ which can be rewritten as
%$$a\left((b'_{i})^{p_i^{s_i-k_i-1}}\right)^{p_i^{k_i}} \equiv 1(mod~p_i^{k_i+1}).$$
%Therefore the choice of $a'_i = (b'_{i})^{p_i^{s_i-k_i-1}}$ works for proving the claim in this case.\\
%Now consider the case when $k_i = s_i$. Since in this case $s_i \leq r_i$, in view of $(\ref{eq:I2})$, we have 
%\begin{equation}\label{eq:l7}
%a^{p_{i}^{s_i}} \equiv a~(mod~p_i^{s_i+1}).
 %\end{equation}
%Also by virtue of $(\ref{eq:l6'}$, we obtain
%$$ab'_i \equiv 1~(mod~p_i^{s_i+1}).$$
%Taking $p_i^{s_i}$th power and using $(\ref{eq:l7})$, we see that
%$$a(b'_i)^{p_i^{s_i}} \equiv ~1~(mod~p_i^{s_1+1}),$$
%which proves our claim.\\

Let $j$ be an integer such that $1 \leq j \leq p^k$. Fix a non-negative integer $w$ for which  $p^{w} \leq j  < p^{w+1}$. By a basic result, the highest power of $p$ dividing $\binom{p^k}{j}$ is $k-v_p(j)$, where $v_p(j)$ stands for the highest power of $p$ dividing $j$.  So $p^{k-w}$ divides $\binom{p^k}{j}$; consequently by virtue of $(\ref{eq:3.31a})$ the lemma is proved once we show that $jk \leq k - w + (j - 1)(k + 1),$ i.e., $0\leq j - w - 1$, which clearly holds in view of the choice of $w$.
\end{proof}

\noindent\textit{Proof of Lemma \ref{3.4}.}
As pointed out in the paragraph preceding Lemma \ref{3.3}, $\dfrac{\TNPMJA}{\CNPMJA}$ is an algebraic integer and hence  the lemma is proved when $k_m = 0$.  From now on, it may be assumed that $k_m \geq 1$. Denote $ \sum\limits_{j=0}^{p^{k_m}-1}\left(a'_m\theta^{n'_m}\right)^{j}$ by $\eta_m$. By hypothesis $p\nmid a$  so $p\nmid C_m a'_m$.   Let $w'_m$ be an integer such that  \begin{equation}\label{aaa}
w_{m}\CNPMJA (a'_{m})^{p^{k_m} - 1} + w'_{m}p^{k_m} = 1.
\end{equation} By Lemma \ref{3.3}, $\dfrac{\eta_m}{p^{k_m}}$ is an algebraic integer and hence so is  $w_{m}\theta^{j_m} \left(\dfrac{\eta_m}{p^{k_m}}\right) + w'_{m}\left(\dfrac{\TNPMJA}{\CNPMJA}\right) = \beta$ (say). Using (\ref{aaa}) together with the fact that $j_m + n'_{m} (p^{k_m} -1) =j_m +  (n-\frac{n}{p^{k_m}} ) =m$, a simple calculation shows that 
$$\beta
= \dfrac{\TNPMJA + \ANPMJA}{p^{k_m}~ \CNPMJA},$$ where $\ANPMJA$ =  $ w_{m}\CNPMJA\theta^{j_{m}} \sum\limits_{j=0}^{p^{k_m}-2}\left(a'_m\theta^{n'_m}\right)^{j} $. This proves the lemma.

\section{Proof of Theorems \ref*{3.3'}, \ref*{3.1} and \ref*{3.new}.}
For algebraic integers $\alpha_1, \alpha_2, \alpha_3, \cdots, \alpha_n$ in $K$ which are linearly independent over $\Q$, $D_{K/\mathbb{Q}} (\alpha_1, \alpha_2, \cdots, \alpha_n)$ will stand for the determinant of the $n\times n$ matrix with $(i,j)$-th entry $Tr_{K/\mathbb{Q}}(\alpha_i \alpha_j)$. If $M$ denotes the group with basis $\alpha_1, \alpha_2, \cdots, \alpha_n $, then as is well known $D_{K/\mathbb{Q}} (\alpha_1, \alpha_2, \cdots, \alpha_n) =[A_K: M]^2 d_K$. So the problem of finding an integral basis of $K$ is same as finding algebraic integers $\gamma_1,\gamma_2,\cdots, \gamma_n$ in $K$ such that $D_{K/\mathbb{Q}} (\gamma_1,\gamma_2,\cdots, \gamma_n) = d_K$. 
We shall use the following elementary result proved in \cite[Problem 435]{Book}.

\noindent\textbf{Lemma 3.A.} {\it Let $t,n$ be positive integers. Then  $$\sum\limits_{m=1}^{n-1} \bigg\lfloor \dfrac{tm}{n} \bigg\rfloor  = \dfrac{1}{2} [(n-1)(t-1) + gcd(t,n) -1].$$}

\noindent\textit{Proof of Theorem \ref{3.1}.} We first prove the theorem when $S = \emptyset$. As pointed out in the paragraph preceding Lemma \ref{3.3}, $\frac{\theta^m}{C_m}$ is an algebraic integer for $1 \leq m \leq n-1$.  The transition matrix from $\{1, \frac{\theta}{C_1}, \frac{\theta^2}{C_2}, \cdots, \frac{\theta^{n-1}}{C_{n-1}}\}$ to $\{1, \theta, \cdots, \theta^{n-1}\}$ has determinant $\prod\limits_{m=1}^{n-1}C_m = C$ (say). As is well known $$ D_{K/\mathbb{Q}}(1,\theta, \theta^2,\cdots, \theta^{n-1})= (-1)^{\binom{n}{2}}N_{K/\mathbb{Q}}(n\theta^{n-1})= (-1)^{\frac{(n-1)(n-2)}{2}}n^n a^{n-1}.$$
Consequently using  a basic result (cf. \cite[Proposition 2.9]{Nar}) and the above equation, we see that 
\begin{equation}\label{eq:3.000}
D_{K/\Q}\left(1, \frac{\theta}{C_1}, \frac{\theta^2}{C_2}, \cdots, \frac{\theta^{n-1}}{C_{n-1}}\right) = \frac{1}{C^2}D_{K/\Q}(1, \theta. \cdots, \theta^{n-1}) = \frac{(-1)^{\frac{(n-1)(n-2)}{2}}n^na^{n-1}}{C^2}.
\end{equation}
Applying Lemma 3.A, we have
\begin{equation}\label{eq:3.41}
C  = \prod\limits_{m=1}^{n-1}\prod\limits_{i=1}^{n-1}a_i^{\lfloor\frac{im}{n}\rfloor} = \prod\limits_{i=1}^{n-1}a_i^{\sum\limits_{m=1}^{n-1}\lfloor\frac{im}{n}\rfloor}  = \prod\limits_{i=1}^{n-1}a_i^{\frac{(n-1)(i-1) + \gcd(i,n)-1}{2}}.
\end{equation}
Substituting for $C$ in $(\ref{eq:3.000})$, we obtain
$$D_{K/\Q}\left(1, \frac{\theta}{C_1}, \frac{\theta^2}{C_2}, \cdots, \frac{\theta^{n-1}}{C_{n-1}}\right) = \frac{(-1)^{\frac{(n-1)(n-2)}{2}}  n^n a^{n-1}}{\prod\limits_{i=1}^{n-1}a_i^{(n-1)(i-1) + \gcd(i,n)-1}}.$$
Keeping in mind that $a = sgn(a)\prod\limits_{i=1}^{n-1}a_i^i$, a simple calculation shows that the above equation can be rewritten as 
\begin{equation}\label{eq:3.111}
D_{K/\Q}\left(1, \frac{\theta}{C_1}, \frac{\theta^2}{C_2}, \cdots, \frac{\theta^{n-1}}{C_{n-1}}\right) = (-1)^{\frac{(n-1)(n-2)}{2}}sgn(a^{n-1})n^n\prod\limits_{i=1}^{n-1}a_i^{n - \gcd(i,n)}. 
\end{equation}
If $\prod\limits_{j=1}^{l}q_j^{t_j}$ is the prime factorization of $|a|$, then it can be easily seen that 
\begin{equation}\label{eq:3.112}
\prod\limits_{i=1}^{n-1}a_i^{n-\gcd(i,n)} =  \prod\limits_{j=1}^{l}q_j^{n-\gcd(t_j, n)}.
\end{equation} 
It is immediate from $(\ref{eq:3.112})$ and Theorem 1.A that the right hand side of $(\ref{eq:3.111})$ equals $d_K$. So we conclude that $\{1, \frac{\theta}{C_1}, \frac{\theta^2}{C_2}, \cdots, \frac{\theta^{n-1}}{C_{n-1}}\}$ is an integral basis of $K$ in the present case.
   
We now deal with the case  $S \neq \emptyset$. Retaining the notation introduced in the paragraph preceding Theorem $\ref{3.1}$, we first show that $\dfrac{\theta^m + \beta_m}{C_m \prod\limits_{i \in S}\PKm}= \gamma_m~ (\text{say})$ is an algebraic integer for $1\leq m\leq n-1$. Fix any $m$, $1\leq m\leq n-1$. If $S_m=\emptyset$, then by definition $\beta_m=0$ and $  \prod\limits_{i \in S}\PKm=1$ and so $\gamma_m= \frac{\theta^m}{C_m}$ is an algebraic integer. Consider now the situation when $S_m\neq \emptyset$.
%Observe that for every pair $(m, i)$ with $0 \leq m \leq n-1$, $1 \leq i \leq t,$ there exist unique $k_{i,m}, j_{i,m}$ such that $m = n_{i}(p_{i}^{s_{i}} - p_{i}^{s_{i}-k_{i,m}})+j_{i,m}$.
%Let $z_{j} = \prod\limits_{i \in T\backslash\{j\}}\PKm$ for $j \in T = \{i : s_i \geq 1\}$.
% In view of the  definition of $z_j$ we have that $gcd\{z_j|j\in T \}=1$ and $gcd\{z_j|j\in T' \}\neq 1$ for any proper subset of $T'$ of $T$, so there exists non-zero $u_i \in \Z$ for $i \in T$ such that $\sum\limits_{i \in T}u_iz_i = 1$.
Recall that when $i \in S_m$, then $z_{i,m} = \prod\limits_{j\in S_m\setminus\{i\}}p_j^{k_{j,m}}$ and $u_{i,m}$ are  integers such that $\sum\limits_{i\in S_m }u_{i,m}z_{i,m}=1$. Further by Lemma \ref{3.4}, the element $\dfrac{\theta^m + \delta_{i,m}}{C_m \PKm }$ of $K$ is an algebraic integer; consequently keeping in mind $\sum_{i\in S_m} u_{i,m} z_{i,m}=1$, we see that $$\sum\limits_{i\in S_m} u_{i,m}\dfrac{\theta^m + \delta_{i,m}}{C_m \PKm} =  \dfrac{\theta^m + \sum\limits_{i\in S_m} u_{i,m}z_{i,m} \delta_{i,m}}{C_m \prod\limits_{i\in S_m} \PKm }  = \dfrac{\theta^m + \beta_m}{C_m \prod\limits_{i \in S_m}\PKm}=\dfrac{\theta^m + \beta_m}{C_m \prod\limits_{i \in S}\PKm}  =\gamma_m $$ is an algebraic integer. 

Taking $\gamma_0=1$, it remains to be shown that $D_{K/\mathbb{Q}}(\gamma_0,\gamma_1,\cdots, \gamma_{n-1})=d_K$. As pointed out at the end of  the paragraph preceding Theorem \ref{3.1}, the power of $\theta$ occuring in $\beta_m$ is less than $m$. So it is clear that  the transition matrix %Now we prove that the set $ \left\lbrace \gamma_m 
%= \dfrac{\theta^m +\beta_m}{C_m \prod\limits_{i\in T} p_i ^{k_{i,m}}} 
%: 0\leq m\leq n-1  \right\rbrace $ is an  integral basis of $K$.  It is clear that the transition matrix from
from  $\{\gamma_0, \gamma_1, \cdots, \gamma_{n-1}\}$ to $\{1, \theta, \cdots, \theta^{n-1}\}$ has determinant $ 
 \prod\limits_{m=1}^{n-1}\left(C_m  \prod\limits_{i\in S}  p_i ^{k_{i,m}} \right)$ = $C\prod\limits_{i\in 
 S}p_i^{\left(\sum\limits_{m=1}^{n-1}k_{i,m}\right)}$, where $C = \prod\limits_{m=1}^{n-1}C_m$. 
Recall that $k_{i,m}$ is the largest integer less than or equal to $d_i$ such that $m\geq n-\frac{n}{p_i ^{k_{i,m}}} $. So when $1\leq j \leq d_i-1$ and $n-\frac{n}{p_i^{j}} \leq m < n - \frac{n}{p_i^{j+1}}$, we have $k_{i,m} = j$ and $k_{i,m} = d_i$ when $n - \frac{n}{p_i^{d_i}} \leq m <n$; consequently 
  \begin{equation*} \label{eq:132}
 \sum\limits_{m=1}^{n-1} k_{i,m} =  \sum\limits_{j=1}^{d_i-1 } j  
\left(  \frac{n}{p_{i}^{j}} -\frac{n}{p_{i}^{j+1}}\right)  + d_i \frac{n}{p_i ^{d_i}}  = \sum\limits_{j=1}^{d_i }\frac{n}{p_i 
 ^j} =n_i \sum\limits_{j=1}^{d_i} p_i ^{s_i -j}.   
 \end{equation*}  
Using the above equation together with $(\ref{eq:3.41})$, we see that the determinant of transition matrix from $\{\gamma_0, \gamma_1, \cdots, \gamma_{n-1}\}$ to $\{1, \theta, \cdots, \theta^{n-1}\}$ equals 
 $\prod\limits_{j=1}^{n-1}a_j^{\frac{(n-1)(j-1) + \gcd(j,n)-1}{2}}\prod\limits_{i\in S}p_i^{n_i 
 \sum\limits_{j=1}^{d_i} p_i ^{s_i -j}}$. Now arguing exactly as in the previous case, it can be easily seen that
%    Now using a basic result (cf. \cite[Proposition 2.9]{Nar}), we obtain
 %$$D_{K/\Q}(1, \theta, \cdots, \theta^{n-1}) = \left(  \prod\limits_{m=1}^{n-1} \prod\limits_{i\in S} C_m  p_i ^{k_{i,m}} \right)^2 D_{K/\Q}(\gamma_0, \gamma_1, \cdots, \gamma_{n-1}).$$
% Using $D_{K/\Q}(1,\theta,\cdots,\theta^{n -1}) = (-1)^{\binom{n-1}{2}} n^{n} a^{n -1}$ in the above equation we see that
% $$D_{K/\Q}(\gamma_0, \gamma_1, \cdots, \gamma_{n-1}) = \frac{(-1)^{\frac{(n-1)(n-2)}{2}} sgn(a^{n-1})n^n\prod\limits_{j=1}^{n-1}a_j^{j(n-1)}}{  \prod\limits_{j=1}^{n-1}a_j^{(n-1)(j-1) + gcd(j,n)-1}\prod\limits_{i\in S}p_i^{2n_i \sum\limits_{j=1}^{d_i} p_i ^{s_i -j}} },$$
% Keeping in mind (\ref{eq:3.41}),(\ref{eq:132}) and the fact that   $a = sgn(a)\prod\limits_{j=1}^{n-1}a_j^j$, a simple calculation shows that the above equation can be written as 
%which can be rewritten as
 $$\vspace*{-0.1in} D_{K/\Q}(\gamma_0, \gamma_1, \cdots, \gamma_{n-1}) = (-1)^{\frac{(n-1)(n-2)}{2}}sgn(a^{n-1})(\prod\limits _{i=1}^k p_i ^{ v_i})\prod\limits_{j=1}^{n-1}a_j^{n - \gcd(j,n)}, \vspace*{-0.1in} $$
  where   $v_i $ equals  $ns_i -2n_i \sum\limits_{j=1}^{d_i }  p_i^{s_i 
 -j}$ or $ns_i $ according as  $i\in S$ or not. In view of $(\ref{eq:3.112})$ and  Theorem 1.A, the right hand side of the above equation equals $d_K$. Therefore  $\{\gamma_0, \gamma_1, \cdots, \gamma_{n-1}\}$ is an integral basis of $K$. 
 %Also using Lemma 2.C, the transition matrix from $\{\gamma_0, \gamma_1, \cdots, \gamma_{n-1}\}$ to $\{1, \theta, \cdots, \theta^{n-1}\}$ has determinant  $ \prod\limits_{m=1}^{n-1} C_m (\prod\limits_{i\in T} p_i ^{k_{i,m}})$ which equals $[A_K:\Z[\theta]]$ in view of (\ref{eq:1102}) and (\ref{eq:132}). This completes the proof of the theorem. % Now using a basic result (cf. \cite[Proposition 2.9]{Nar}), we see that
 %\begin{equation} \label{eq:131}
%D_{K/\mathbb{Q}} (1,\theta,\cdots, \theta^{n-1}) = \left( \prod\limits_{m=1}^{n-1} C_m (\prod\limits_{i\in T} p_i ^{k_{i,m}}) \right)^{2}  D_{K/\mathbb{Q}} (\gamma_0, \gamma_1, \cdots, \gamma_{n-1}). \end{equation}  
%Keeping in mind $(\ref{eq:132})$ and comparing the above equation with $(\ref{eq:2603})$, we see that
%$$ D_{K/\mathbb{Q}} (\gamma_0, \gamma_1, \cdots, \gamma_{n-1}) = \pm d_K $$ which completes the proof of the theorem.
\\

\noindent\textit{Proof of Theorem \ref{3.3'}.} It is immediate from the definition of $k_m$ that $k_m\geq 1$ if and only if  $ m\geq\phi(p^s)$; consequently (with notations as in Theorem $\ref{3.1}$) $\beta_m=0$ for $0\leq m < \phi(p^s)$  and hence the corollary follows at once from Theorem \ref{3.1} and Remark \ref{3.5'}.\hspace*{14.9cm}$\square$

\begin{proof}[Proof of Theorem \ref{3.new}.] We retain the notation of Theorem \ref{3.1} and adopt the convention that $\beta_m = 0$  for each $m$ if $S$ is empty. Since $a$ is squarefree, $C_m = 1$ for $1\leq m\leq n-1.$ So in view of   Theorem \ref{3.1},  $ \lbrace 1,  \dfrac{\theta^m +\beta_m}{ \prod\limits_{i\in S} p_i ^{k_{i,m}}} ~\big\vert~ 1\leq m\leq n-1  \rbrace $ is an integral basis of $K$. Let $\beta'_m$ denote the element of $K'$ obtained on replacing $\theta$ by $\theta'$ in the expression for $\beta_m$. The theorem is proved once we show that $ \lbrace 1,  \dfrac{\theta'^{m} +\beta'_m}{ \prod\limits_{i\in S} p_i ^{k_{i,m}}} ~\big\vert~ 1\leq m\leq n-1  \rbrace $ is an integral basis of $K'$. Fix any prime $p$ dividing $n$ and let $s$ denote the highest power of $p$ dividing $n$. Let $r, r'$ stand respectively for $v_p(a^{p-1}-1)-1, v_p(a'^{p-1}-1)-1$. In view of the definition of $\beta_m$, the desired assertion is proved once we show that
\begin{equation}\label{ans}
\min\{r, s\} = \min\{r', s\}.
\end{equation}
By hypothesis $a' \equiv a~(mod~np)$. Note that $(\ref{ans})$ needs to be verified when $p\nmid a$. Write $a'  = a+p^{s+1}b,$ $b\in \Z$.  So there exists $c \in \Z$ such that 
\begin{equation}\label{anssu}
r' + 1 = v_p(a'^{p} - a') = v_p(a^{p} - a + p^{s+1}c).
\end{equation}
If $r < s$, then  the above equation shows that $ r' + 1 = r +1$, which proves $(\ref{ans})$ in this case. %If $r = s$, then $(\ref{anssu})$ shows that $r' + 1 \geq s+1,$ which yields $(\ref{ans})$ in this case.
 If $r \geq s$, then by $(\ref{anssu})$, $r' + 1 \geq s+1$ and hence  $(\ref{ans})$ is verified. This completes the proof of the theorem.
\end{proof}

\noindent\textbf{Acknowledgement.}  The second author is thankful to Indian National Science Academy, New Delhi for fellowship.

%\noindent Define $\Phi_{p,q} : v(\xi) \rightarrow v(\xi_{p,q})$\\
%\hspace*{1.1in}$v_{\xi} \mapsto v_{\xi_{p,q}}$


\begin{thebibliography}{13}
\baselineskip 13pt
%\bibitem{HN2} S. Ahmed, T. Nakahara, A. Hameed, On certain pure sextic fields related to a problem of Hasse, Int. J. Algebra and Computation, Vol. 26, No. 3 (2016) 577-583.
%\bibitem{AN} S. Ahmed, T. Nakahara, S. M. Husnine, Power integral bases for certain pure sextic fields, Int. J.  Number Theory, Vol. 10, No. 8 (2014) 2257–2265.
%\bibitem{Ber} W. E. H. Berwick, Integral basis, Cambridge University Press, London, 1927.

 \bibitem{Ded} R. Dedekind, \"{U}ber die Anzahl der Idealklassen in reinen kubischen Zahlk\"{o}rpern, Journal f\"{u}r die reine und angewandte Mathematik 121 (1900), 40-123.
% \bibitem [Coh]{Coh} H. Cohen, ``A Course in Computational Algebraic Number Theory", Springer-Verlag, Berlin Heidelberg, 1993.
 \bibitem{Fun} T. Funakura, On integral bases of pure quartic fields, Math. J. Okayama Univ, 26 (1984), 27-41.
 \bibitem{Ga-Re} I. Ga\'{a}l, L. Remete, Integral basis and monogenity of pure number fields, J. Number Theory, {173} (2017), 129-146. 
 \bibitem{Gass} T. A. Gassert, A note on the monogenity of power maps, Albanian J. Math. 11 (2017), 3-12. 
 \bibitem{HN} A. Hameed, T. Nakahara, Integral bases and relative monogenity of pure octic fields, Bull. Math. Soc. Sci. Math. Roumanie
 Tome 58(106) No. 4 (2015), 419–433.
% \bibitem{Spe} J. G. Huard, Blair K. Spearman, and Kenneth S. Williams, Integral bases for quartic fields with quadratic subfields, J.  Number Theory, 51 No. 1 (1995) 87-102.
% \bibitem{JKS27} A. Jakhar, S. K. Khanduja, N. Sangwan, Discriminant of pure squarefree degree number fields, Acta Arith., 181, 2017, 287-296.
%  \bibitem{SKS1}  S. K. Khanduja, S. Kumar, On prolongations of valuations via Newton polygons and liftings of polynomials,   J. Pure Appl. Algebra, 216, 2012, 2648-2656.
 
% \bibitem {SKS2} S. K. Khanduja, S. Kumar, A generalization of a theorem of Ore,  J. Pure Appl. Algebra,  218, 2014, 1206-1218.
 
%  \bibitem{Book} J. M. de Koninck, A. Mercier, 1001 Problems in Classical Number Theory, American  Mathematical Society, Providence Rhode Island, 2007. 
 
% \bibitem{Lan} G. Landsberg, Ueber das Fundamentalsystem und die Discriminante der Gattungen algebraischer Zahlen, welche aus Wurzelgr\"{o}ssen gebildet sind, Journal f\"{u}r die reine und angewandte Mathematik, 117, 1897, 140-147 .
% \bibitem{Mo-Na} J. Montes, E. Nart, On a Theorem of Ore, {J. Algebra,} {146}, 1992, 318-334.
\bibitem{JKS26} A. Jakhar, S.K. Khanduja, N. Sangwan, On the discriminant of pure number fields, 2020, http://arxiv.org/abs/2005.01300.
 \bibitem{Book} J. M. de Koninck, A. Mercier, 1001 Problems in Classical Number Theory, American  Mathematical Society, Providence Rhode Island, 2007.
 \bibitem{Nar} W. Narkewicz, Elementary and Analytic Theory of Algebraic Numbers, Springer-Verlag, Berlin Heidelberg, 2004. 
%\bibitem{KOK1} K. Okutsu, Integral Basis of the Field $\Q(\sqrt[n]{a})$ I, Proc. Japan Acad., 58A, 1982,  47-49.
% \bibitem{KOK2} K. Okutsu, Integral Basis of the Field $\Q(\sqrt[n]{a})$ II, Proc. Japan Acad., 58A, 1982, 87-89.
% \bibitem{KOK3} K. Okutsu, Integral Basis of the Field $\Q(\sqrt[n]{a})$ III, Proc. Japan Acad., 58A, 1982, 117-119.
% \bibitem{KOK4} K. Okutsu, Integral Basis of the Field $\Q(\sqrt[n]{a})$ IV, Proc. Japan Acad., 58A, 1982, 167-169.
% \bibitem{KOK} K. Okutsu, Integral Basis of the Field $\Q(\sqrt[n]{a})$, Proc. Japan Acad., 58A, 1982, 219-222.
% \bibitem{Ore} {\O}. Ore, Newtonsche Polygone in der Theorie der algebraischen K\"{o}rper, {Math. Ann.,} {99}, 1928, 84-117.
 \bibitem{Wes} J. Westlund, On the fundamental number of the algebraic number field $k(\sqrt[p]{m})$, Trans. Amer. Math. Soc. 11 (1910), 388-392.
 



\end{thebibliography}
\end{document}